\documentclass[10pt]{article}

\usepackage{amssymb}
\usepackage{amscd,enumerate,amsfonts,calc,amsmath,verbatim,xypic}
\usepackage[all]{xy}

\usepackage{amsthm}

\providecommand{\hhalf}{\mbox{\hspace{0.5em}}}

\theoremstyle{plain}
\newtheorem{theorem}{Theorem}[section]
\newtheorem{corollary}[theorem]{Corollary}
\newtheorem{proposition}[theorem]{Proposition}
\newtheorem{question}[theorem]{Question}
\newtheorem{lemma}[theorem]{Lemma}
\newtheorem{remark}[theorem]{{Remark}}

\newtheorem{definition}[theorem]{Definition}
\newtheorem{example}[theorem]{Example}

\providecommand{\w}{\omega}
\providecommand{\lm}{\lambda}
\providecommand{\rar}{\rightarrow}
\providecommand{\lrar}{\longrightarrow}

\providecommand{\inc}{\subseteq}

\renewcommand{\vec}[3]{\ensuremath{#1_#2, \ldots, #1_#3 }}

\renewcommand\char{\text{\rm char}}

\providecommand\Hom{\text{\rm Hom}}

\providecommand\im{\text{\rm im}}
\providecommand\ker{\text{\rm ker}}
\providecommand\Max{\text{\rm Max}}
\providecommand\min{\text{\rm min}}
\providecommand\ord{\text{\rm ord}}
\providecommand\soc{\text{\rm soc}}

\renewcommand\-{{\_\!\_}}
\providecommand{\sk}{{\ensuremath{\sf k }}}
\providecommand{\ma}{{\mathfrak a }}
\providecommand{\mb}{{\mathfrak b }}
\providecommand{\mc}{{\mathfrak c }}
\providecommand{\md}{{\mathfrak d }}

\providecommand{\m}{{\mathfrak m }}
\providecommand{\n}{{\mathfrak n }}

\begin{document}

\title{\large Computing Gorenstein Colength }
\author{H. Ananthnarayan}
\date{October 6, 2008}
\maketitle

\abstract{
\noindent
Given an Artinian local ring $R$, we define (in \cite{A}) its Gorenstein colength $g(R)$ to measure how closely we can approximate $R$ by a Gorenstein Artin local ring.  In this paper, we show that $R = T/\mb$ satisfies the inequality $g(R) \leq \lambda(R/\soc(R))$  in the following two cases: (a) $T$ is a power series ring over a field of characteristic zero and $\mb$ an ideal that is the power of a system of parameters or (b) $T$ is a 2-dimensional regular local ring with infinite residue field and $\mb$ is primary to the maximal ideal of $T$.

In the first case, we compute $g(R)$ by constructing a Gorenstein Artin local ring mapping onto $R$. We further use this construction to show that an ideal that is the $n$th power of a system of parameters is directly linked to the $(n-1)$st power via Gorenstein ideals. A similar method shows that such ideals are also directly linked to themselves via Gorenstein ideals.\\
Keywords: Gorenstein colength; Gorenstein linkage.
}

\section{\large Introduction}

Let us first recall the definition of Gorenstein colength and review some of its basic properties from \cite{A} in this section.

\begin{definition}\label{D0}{\rm 
Let $(R,\m,\sk)$ be an Artinian local ring. Define the Gorenstein colength of $R$, denoted $g(R)$ as:\\
$g(R) = \min\{\lambda(S) - \lambda(R) : S$ is a Gorenstein Artin local ring 
mapping onto $R\},$\\ where $\lambda(\-)$ denotes length. 
}\end{definition}

The main questions one would like to answer are the following: 

\begin{question}\label{MQ}\hfill{}\\{\rm
a) How does one intrinsically compute $g(R)$? 

\noindent
b) How does one construct a Gorenstein Artin local ring $S$ mapping onto $R$ such that $\lambda(S) - \lambda(R) = g(R)$?
}\end{question}

In order to answer Question \ref{MQ}(a), we prove the following inequalities in \cite{A}, which give bounds on $g(R)$.

\[
\begin{tabular}{|c|}
\hline
\\
$\lambda(R/(\omega^*(\omega))) \leq\hhalf \min\{\lambda(R/\ma): \ma \text{ is an ideal in } R, \ma \simeq \ma^\vee\}\hhalf \leq\hhalf  g(R)\hhalf \leq\hhalf  \lambda(R),$\\
\\
\text{\bf Fundamental Inequalities}
\\
\hline
\end{tabular}
\]
where $\w$ is the canonical module of $R$, $\w^*(\w) = \langle f(\w): f \in \Hom_R(\w,R) \rangle$ is the trace ideal of $\w$ in $R$ and $\ma^\vee = \Hom_R(\ma,\omega)$. 

A natural question one can ask in this context is Question 3.10 in \cite{A}, which is the following:

\begin{question}{\rm
Let ($R,\m,\sk)$ be an Artinian local ring and $\ma$ an ideal in $R$ such that $\ma \simeq \ma^\vee$. Does there exist a Gorenstein Artinian local ring $S$ mapping onto $R$, such that $\lm(S) - \lm(R) = \lm(R/\ma)$?
}\end{question}

The socle of $R$, \soc(R),  is a direct sum of finitely many copies of $\sk$, hence it is isomorphic to $\soc(R)^\vee$. Hence a particular case of the above question is the following:

\noindent
\begin{question}\label{Q1}{\rm
Is there a Gorenstein Artin local ring $S$ mapping onto $R$ such that $\lambda(S) -\lambda(R) = \lambda(R/\soc(R))$?
}\end{question}

A weaker question one can ask is the following:

\begin{question}\label{Q2}{\rm
Is $g(R) \leq \lm(R/\soc(R))$?
}\end{question}

We answer Question \ref{Q2} in two cases in this paper. In section 3,  we show that if $T$ is a power series ring over a field and $\md = (f_1,\ldots,f_d)$ is an ideal generated by a system of parameters, then $g(T/\md^n) \geq \lm(T/\md^{n-1})$. Further, if the residue field of $T$ has characteristic zero, we construct a Gorenstein Artin local ring $S$ mapping onto $T/\md^n$ such that $\lm(S) - \lm(T/\md^n) = \lm(T/\md^{n-1})$ using a theorem of L. Ried, L. Roberts and M. Roitman proved in \cite{RRR}. This shows that $g(T/\md^n) = \lm(T/\md^{n-1})$. In particular, this proves that $R = T/\md^n$ satisfies the inequality in Question \ref{Q2}.

In \cite{KMMNP}, Kleppe, Migliore, Miro-Roig, Nagel and Peterson show that $\md^n$ can be linked to $\md^{n-1}$ via Gorenstein ideals in 2 steps and hence to $\md$ in $2(n-1)$ steps. In section 4, we use the ideal corresponding to the Gorenstein ring constructed in section 3, to show that $\md^n$ can  be directly liked to $\md^{n-1}$ and hence to $\md$ in $(n-1)$ steps.

When $R$ is an Artinian quotient of a two-dimensional regular local ring with an infinite residue field, we use a  formula due to Hoskin and Deligne (Theorem \ref{HD}) in order to answer Question \ref{Q2} in section 5.

\section{\large Computing ${\mathbf \w^*(\w)}$}

Let $(R,\m,\sk)$ be an Artinian local ring with canonical module $\w$. As noted in \cite{A}, maps from $\w$ to $R$ play an important role in the study of Gorenstein colength. In this section, we prove a lemma which helps us compute the trace ideal $\w^*(\w)$ of $\w$ in $R$. We use the following notation in this section.\\

\noindent
{\bf Notation:} Let $(T,\m_T,\sk)$ be a regular local ring mapping onto $R$. Let $$0 \rar T^{b_d} \overset{\phi}\rar T^{b_{d-1}}\rar \ldots \rar T \rar R \rar 0 \quad\quad\quad(\sharp)$$ be a minimal resolution of $R$ over $T$. Then a resolution of the canonical module $\w$ of $R$ over $T$ is given by taking the dual of the above resolution, i.e., by applying $\Hom_T(\-,T)$ to the above resolution. Hence a presentation of $\w$ is $T^{b_{d-1}} \overset{\phi^*}\rar T^{b_d} \rar \w \rar 0$. Tensor with $R$ and apply $\Hom_R(-,R)$ to get an exact sequence $0 \lrar \w^* \lrar R^{b_d} \overset{\phi \otimes R}\lrar R^{b_{d-1}}$. Let $\w^*$ be generated minimally by $b_{d+1}$ elements. Thus we have an exact sequence $R^{b_{d+1}} \overset{\psi} \lrar R^{b_d} \overset{\phi \otimes R}\lrar R^{b_{d-1}}$, where $\w^* = \im(\psi)$.  

\begin{lemma}\label{L1}
With notation as above, let $\psi$ be given by the matrix $(a_{ij})$. Then the trace ideal of $\w$, $\w^*(\w)$, is the ideal generated by the $a_{ij}$'s.
\end{lemma}

The above lemma is a particular case of the following lemma. 

\begin{lemma}\label{L2}
Let $(R,\m,\sk)$ be a Noetherian local ring and $M$ a finitely generated $R$-module. Let $\ R^n \overset{B}\lrar R^m \lrar M \lrar 0$ be a minimal presentation of $M$. Apply $\Hom_R(\-,R)$ to get an exact sequence $0 \lrar M^* \lrar (R^*)^m \overset{B^*}\lrar (R^*)^n$. Map a free $R$-module, say $R^k$, minimally onto $M^*$ to get an exact sequence $R^k \overset{A}\lrar (R^*)^m \overset{B^*}\lrar (R^*)^n$, where $M^* = \ker(B^*) = \im(A)$. Then the trace ideal of $M$, $M^*(M) = (a_{ij}: a_{ij}\text{ are the entries of the matrix }A)$.
\end{lemma}

\begin{proof}
Let $\vec{m}{1}{n}$ be a minimal generating set of $M$, $\vec{e}{1}{m}$ be a basis of $R^m$ such that $e_i \mapsto m_i$, and $\vec{e^*}{1}{m}$ be the corresponding dual basis of $(R^*)^m$. 

Let $f \in M^*$. Write $f = \Sigma_{i=1}^m r_i e^*_i \in (R^*)^m$. Then $f$ acts on $M$ by sending $m_j$ to $r_j$. Hence if $A = (a_{ij})$, then the generators of $M^*$ are $f_j = \Sigma_{i=1}^m a_{ij} e^*_i$, $1 \leq j \leq k$. Thus $f_j(m_i) = a_{ij}$. Thus $M^*(M) = (a_{ij})$.
\end{proof}

\begin{corollary}\label{C0}
With notation as above, let $(T',\m_{T'},\sk)$ be a regular local ring which is a flat extension of $T$ such that $\m_T T' \inc \m_{T'}$ and let $R' = T' \otimes_T R$. Then $\w_{R'}^*(\w_{R'}) = \w^*(\w)T'$.
\end{corollary}

\begin{proof}
Since $T'$ is flat over $T$, $R' = T' \otimes_{T} R$ and $\m_T T' \inc \m_{T'}$, a minimal resolution of $R'$ over $T'$ is obtained by tensoring ($\sharp$) by $T'$ over $T$. Therefore $\w_{R'}^*(\w_{R'})$ is the ideal generated by the entries of the matrix $\psi \otimes_{T} T'$. Now, by Lemma \ref{L1}, the ideal in $R$ generated by the entries of $\psi$ is $\w^*(\w)$. Therefore, $\w_{R'}^*(\w_{R'}) = \w^*(\w) T'$.
\end{proof}

\section{\large Powers of Ideals Generated by a System of Parameters}

In this section, the main theorem we prove is the following:

\begin{theorem}\label{sop}
Let $T = \sk[|\vec{X}{1}{d}|]$ be a power series ring over a field $\sk$ of characteristic zero. Let $\vec{f}{1}{d}$ be a system of parameters in $T$ and $R = T/(\vec{f}{1}{d})^n$. Then $g(R) = \lm(T/(\vec{f}{1}{d})^{n-1})$. 
\end{theorem}

In order to prove this, we first prove the theorem when $f_i = X_i, i = 1,\ldots, d,$ and then use the fact that $T$ is flat over $T' = \sk[|\vec{f}{1}{d}|]$.

\begin{theorem}\label{max2}
Let $T = \sk[|\vec{X}{1}{d}|]$ be a power series ring over a field $\sk$ with maximal ideal $\m_T = (\vec{X}{1}{d})$. Let $R = T/\m_T^n$ and $\w$ be the canonical module of $R$. Then $\w^*(\w) = \soc(R) = \m_T^{n-1}/\m_T^n$.
\end{theorem}

\begin{proof}
In order to prove this, we show that if $\phi \in \Hom(\w,R)$, then $\phi(\w) \inc \soc(R)$. Since $\soc(R) \inc \w^*(\w)$, this will prove the theorem. 

Note that we can consider $R$ to be the quotient of the polynomial ring $\sk[\vec{X}{1}{d}]$ by $(\vec{X}{1}{d})^n$. Thus change notation so that $T = \sk[\vec{X}{1}{d}]$ and $\m_T = (\vec{X}{1}{d})$ is its unique homogenous maximal ideal. 

The injective hull of $\sk$ over $T$, $E_T(\sk)$, is $\sk[\vec{X^{-1}}{1}{d}]$, where the multiplication is defined by \[(X_1^{a_1}\cdots X_d^{a_d}) \cdot (X_1^{-b_1}\cdots X_d^{-b_d}) = \left\{\begin{array}{cc} X_1^{a_1 - b_1} \cdots X_d^{a_d - b_d} & \text{ if }a_i \leq b_i \text{ for all }i\\ 0& \text{ otherwise}\end{array} \right.\]  and extended linearly (e.g., see \cite{N}).

Let $\mb = \m_T^n$. The canonical module $\w$ of $R$ is isomorphic to the injective hull of the residure field of $R$. Hence $\w \simeq \Hom_R(R,E_T(\sk)) \simeq (0 :_{\sk[\vec{X^{-1}}{1}{d}]} \mb)$. Note that $\mb \cdot (X_1^{-a_1}\cdots X_d^{-a_d}) = 0$ whenever $a_i \geq 0$ and $n > \sum a_i$. Since $\lm(\w) = \lm(R)$, we conclude that $$\w \simeq \sk\text{-span of }\left\{X_1^{-a_1}\cdots X_d^{-a_d}: a_i \geq 0; n > \sum_{i=1}^d a_i\right \}.$$
Observe that $\w$ is generated by $\{X_1^{-a_1}\cdots X_d^{-a_d}$: $\sum_{i=1}^d a_i = n-1\}$ as an $R$-module. Let $\phi \in \w^*$. We will now show that $\phi(X_1^{-a_1}\cdots X_d^{-a_d}) \in \soc(R)$ by induction on $a_1$. Let $w = X_1^{-a_1}\cdots X_d^{-a_d}$, $\sum_{i=1}^d a_i = n-1$. 

If $a_1 = 0$, then $X_1 \cdot w = 0$. Hence $\phi(w) \in (0:_R X_1) = \soc(R)$. If not, then $X_1 w = X_2 (X_1^{-(a_1 - 1)}X_2^{-(a_2+1)}\cdots X_d^{-a_d})$. We have $\phi(X_1^{-(a_1 - 1)}X_2^{-(a_2+1)}\cdots X_d^{-a_d})$ $\in \soc(R)$ by induction. Thus $X_2 \phi(X_1^{-(a_1 - 1)}X_2^{-(a_2+1)}\cdots X_d^{-a_d}) = 0$ which yields $X_1 \phi(w) = 0$. But $(0 :_R X_1) = \soc(R)$, which proves that $\phi(\w)\inc \soc(R)$.
\end{proof}

Since we know that $\lm(R/(\w^*(\w)) \leq g(R)$ by the fundamental inequalities, we immediately get the following:

\begin{corollary}\label{C1}
With notation as in Thoerem \ref{max2}, $g(R) \geq \lm(R/\soc(R))$.
\end{corollary}

We prove the reverse inequality in Theorem \ref{max1} by constructing a Gorenstein Artin ring $S$ mapping onto $R$ such that $\lm(S) - \lm(R) = \lm(R/\soc(R))$. The following theorem of Ried, Roberts and Roitman is used in the construction.

\begin{theorem}[Reid,Roberts,Roitman]\label{R3}
Let $\sk$ be a field of characteristic zero, $S = \sk[\vec{X}{1}{d}]/(X_1^{n_1},\ldots,X_d^{n_d}) = \sk[\vec{x}{1}{d}]$. Let $m \geq 1$ and $f$ be a nonzero homogeneous element in $S$ such that $(x_1  + \cdots + x_d)^mf = 0$. Then $\deg(f) \geq (t - m + 1)/2$, where $t = \sum_{i=1}^d(n_i - 1)$.
\end{theorem}

\noindent
We use the following notation in this section.

\noindent
{\bf Notation:} Let $\sk$ be a field. For any graded ring $S$ (with $S_0 = \sk$), by $h_S(i)$ we mean the $\sk$-dimension of the $i$th graded piece of the ring $S$ and if $S$ is Artinian, $\Max(S) := \max\{i:h_S(i) \neq 0\}$. All $\sk$-algebras in this section are standard graded, i.e., they are generated as a $\sk$-algebra by elements of degree 1. \\

\noindent
We also need the following basic fact in order to prove Theorem \ref{max1}.

\begin{remark}\label{R0}{\rm
Let $S = \sk[X_1,\ldots,X_d]/(X_1^{n_1}, \ldots, X_d^{n_d})$ be a quotient of the polynomial ring over a field $\sk$ and $f$ be a non-zero homogeneous element in $S$ of degree $s$. Then $S/(0:_Sf)$ is Gorenstein and $\Max(S/(0:_Sf)) = \Max(S) - s$. 
}\end{remark}

\begin{proposition}\label{P0.5}
Let $T = \sk[\vec{X}{1}{d}]$ be a polynomial ring over $\sk$ and $\m_T = (\vec{X}{1}{d})$ be its unique homogeneous maximal ideal. Let $f$ be a homogeneous element and $\mc = (\vec{X^n}{1}{d}):_Tf$ be such that $\mc \inc \m^n$. Then the following are equivalent:\\
{\rm i)} $\lm(\m_T^n/\mc) = \lm(T/\m_T^{n-1})$. \\
{\rm ii)} $\Max(T/\mc) = 2(n-1)$.\\
{\rm iii)} $\deg(f) = (d-2)(n-1)$.
\end{proposition}

\begin{proof}
Since $\Max(T/(\vec{X^n}{1}{d})) = d(n-1)$, (ii) $\Leftrightarrow$ (iii) follows from Remark \ref{R0}.

Let $R = T/\m_T^n$ and $S = T/\mc$. Since $T/(\vec{X^n}{1}{d})$ is a Gorenstein Artin local ring, so is $S$. Note that $\soc(R) = \m_T^{n-1}/\m_T^n$ and $\lm(S) - \lm(R) = \lm(\m_T^n/\mc)$. 

The rings $R$ and $S$ are quotients of the polynomial ring $\sk[\vec{X}{1}{d}]$ by homogeneous ideals. Thus, both $R$ and $S$ are graded under the standard grading. Since $\mc \inc \m_T^n$, $$h_S(i) = h_R(i)\text{ for }i < n.\quad\quad\quad(\ast)$$

Since $S$ is Gorenstein, $$h_S(i) = h_S(\Max(S) - i).\quad\quad\quad(\ast\ast)$$

Using ($\ast$) and ($\ast\ast$), we see that the Hilbert function of $S$ is:

\begin{center}
\begin{tabular}{|l||c|c|c|c|c|c|}\hline
degree i & 0 & 1 & 2 & 3 & \ldots & n-1 \\ \hline 
$h_R(i)$ & 1 & d & $\left( \begin{array}{c} d + 1 \\ 2\end{array}\right)$ & $\left( \begin{array}{c} d + 2 \\ 3\end{array}\right)$ & \ldots & $\left( \begin{array}{c} d + n - 2 \\ n-1\end{array}\right)$\\ \hline 
$h_S(i)$ & 1 & d & $\left( \begin{array}{c} d + 1 \\ 2\end{array}\right)$ & $\left( \begin{array}{c} d + 2 \\ 3\end{array}\right)$ & \ldots & $\left( \begin{array}{c} d + n - 2 \\ n-1\end{array}\right)$ \\\hline 
\end{tabular}\\\end{center}

\begin{center}
\begin{tabular}{|l||c|c|c|c|c|c|}\hline
degree i &\ldots & \Max(S) - (n-1) & \Max(S) - (n-2)  &\ldots & \Max(S) - 1 & \Max(S)\\ \hline 
$h_R(i)$ & \ldots & 0 & 0 & \ldots & 0 & 0 \\ \hline 
$h_S(i)$ & \ldots & $\left( \begin{array}{c} d + n - 2 \\ n-1\end{array}\right)$  
& $\left( \begin{array}{c} d + n - 3 \\ n-2\end{array}\right)$  & \ldots & d & 1\\ \hline 
\end{tabular}\\\end{center}
Thus we have
\[
\begin{array}{rl}
\lm(T/\m_T^{n-1}) &=  h_R(n-2) + h_R(n-3) + \ldots + h_R(0)\\ & = h_S(n-2) + h_S(n-3) + \ldots + h_S(0)\\ & = h_S(\Max(S) - (n-2)) + h_S(\Max(S) - (n-3)) + \ldots + h_S(\Max(S))\\ &= \sum_{i \geq \Max(S) -(n-2) }h_S(i)\\ &\leq \lm(S) - \lm(R) = \lm(\m_T^n/\mc). 
\end{array}
\] 
Moreover, from the above table, equality holds if and only if $\Max(S) - (n-1) = n-1$, proving (i) $\Leftrightarrow$ (ii). 
\end{proof}

In the following corollary, we show that $f = (X_1 + \cdots + X_d)^{(d-2)(n-1)}$ satisifies the hypothesis of Proposition \ref{P0.5}.

\begin{corollary}\label{P1}
Let $T = \sk[\vec{X}{1}{d}]$ be a polynomial ring over $\sk$, a field of characteristic zero, and $\m_T = (\vec{X}{1}{d})$ be its unique homogeneous maximal ideal. Let $\mc_n = (\vec{X^n}{1}{d}):l^{(d-2)(n-1)}$, where $l = X_1 + \cdots + X_d$. Then $\mc_n \inc \m_T^n$.   

Moreover, $\lm(\m_T^n/ \mc_n) = \lm(T/ \m_T^{n-1})$.
\end{corollary}

\begin{proof}
By Theorem \ref{R3}, if $F$ is a homogeneous element in $T$ such that $l^m F \in (X_1^n,\ldots,X_d^n)$, then $\deg(F) \geq (d(n-1) - m + 1)/2$. Therefore, for $m = (d-2)(n-1)$, we see that $\deg(F) \geq n-1/2$, i.e., $F \in \m_T^n$. Thus $(\vec{X^n}{1}{d}):(X_1 + \cdots X_d)^{(d-2)(n-1)} \inc \m_T^n$.  

Moreover, by Proposition \ref{P0.5}, since $\deg(l^{(d-2)(n-1)}) = (d-2)(n-1)$, $\lm(\m_T^n/ \mc_n) = \lm(T/ \m_T^{n-1})$.
\end{proof}

\begin{theorem}\label{max1}
Let $T = \sk[|\vec{X}{1}{d}|]$ be a power series ring over a field $\sk$ of characteristic zero, with unique maximal ideal $\m_T = (\vec{X}{1}{d})$. Let $R := T/\m_T^n$. Then $g(R) \leq \lm(R/\soc(R))= \lm(T/\m^{n-1})$.
\end{theorem}

\begin{proof}
Let $\mc_n = (\vec{X^n}{1}{d}):_Tl^{(d-2)(n-1)}$, where $l = X_1+\cdots+X_d$. Let $S = T/\mc_n$. Then $S$ is a Gorenstein Artin local ring mapping onto $R$. Note that $R \simeq \sk[\vec{X}{1}{d}]/(\vec{X}{1}{d})^n$ and $S \simeq \sk[\vec{X}{1}{d}]/((\vec{X^n}{1}{d}):_Tl^{(d-2)(n-1)})$.

Hence, by Corollary \ref{P1}, $\lm(S) - \lm(R) = \lm(R/\soc(R)) = \lm(T/\m_T^{n-1})$. This shows that $g(R) \leq \lm(R/\soc(R))$. 
\end{proof}

\begin{remark}{\rm
The ring $S$ constructed in the proof of the theorem does not work when $\char(\sk) = 2$. For example, when $d = 3$ and $n = 3$, we have $h_R(i) = 1,3,6$ and $h_S(i) = 1,2,5,2,1$.
}\end{remark}

\begin{remark}{\rm
Let $S$ be a graded Gorenstein Artin quotient of $T = \sk[X_1,\ldots,X_d]$, where $\sk$ is a field of characteristic zero. We say that $S$ is a compressed Gorenstein algebra of socle degree $t = \Max(S)$, if for each $i$, $h_S(i)$ is the maximum possible given $d$ and $t$, i.e., $h_S(i) = \min\{h_T(i),h_T(t-i)\}$ (e.g., see \cite{IK}). Note that the proofs of Proposition \ref{P0.5} and Corollary \ref{P1} show that $S = T/((\vec{X^n}{1}{d}):_Tl^{(d-2)(n-1)})$ is a compressed Gorenstein Artin algebra of socle degree $2n-2$. A similar technique also shows that $S = T/((\vec{X^n}{1}{d}):_Tl^{(d-2)(n-1) - 1})$ is a compressed Gorenstein Artin algebra of socle degree $2n-1$.
}\end{remark}

In the following remark, we record some key observations which we will use to prove Theorem \ref{sop}.
 
\begin{remark}\label{R}{\rm
Let $T = \sk[|\vec{X}{1}{d}|]$ be a power series ring over a field $\sk$. Let $\vec{f}{1}{d}$ be a system of parameters. Then $T' = \sk[|\vec{f}{1}{d}|]$ is a power series ring and $T$ is free over $T'$ of rank $e = \lm(T/(\vec{f}{1}{d}))$.  Thus, if $\mb$ and $\mc$ are ideals in $T'$, then $(\mc:_{T'}\mb)T = (\mc T :_T \mb T)$ and $\lm(T/\mb T) = e \cdot \lm(T'/\mb)$.
}\end{remark}

Firstly, we construct a Gorenstein Artin ring $S$ mapping onto $R$ such that $\lm(S) - \lm(R) = \lm(T/(\vec{f}{1}{d})^{n-1})$ which proves $g(R) \leq \lm(T/(\vec{f}{1}{d})^{n-1})$. We do this as follows:

Suppose that $\char(\sk) = 0$. Let $\md = (\vec{f}{1}{d})$, $\mc =(\vec{f^n}{1}{d}):_{T'}l^{(d-2)(n-1)}$, where $l = (f_1 + \cdots + f_d)$. We see that since $(\vec{f^n}{1}{d}) :_{T'}  l^{(d-2)(n-1)} \inc \md^n$ in $T'$ by Corollary \ref{P1}, the same holds in $T$  by using Remark \ref{R}. Moreover, since $\lm(\md^nT/\mc T) = e \lm(\md^n/\mc)$ and $\lm(T/\md^{n-1} T) = e \lm(T'/\md^{n-1})$, the length condition in Corollary \ref{P1} gives $\lm(\md^nT/\mc T) =  \lm(T/\md^{n-1} T)$.

This implies that if $R = T/\md^nT$, then $S = T/\mc T$ is a Gorenstein Artin ring mapping onto $R$ and $\lm(S) - \lm(R) = \lm(\md^nT/\mc T) =  \lm(T/\md^{n-1} T)$. Therefore $g(R) \leq \lm(T/\md^{n-1} T)$.  Thus as a consequence of Theorem \ref{max1}, we have proved 

\begin{theorem}\label{sop1}
Let $T = \sk[|\vec{X}{1}{d}|]$ be a power series ring over a field $\sk$ of characteristic zero, $\vec{f}{1}{d}$ be a system of parameters and $\md = (\vec{f}{1}{d})$. Let $R = T/\md^n$. Then $g(R) \leq \lm(T/\md^{n-1})$.
\end{theorem}

In order to prove Theorem \ref{sop}, we know need to show that $g(R) \geq \lm(T/\md^{n-1})$. We prove this by first computing the trace ideal $\w^*(\w)$ of the canonical module and use the fundamental inequalities. We use the lemmas concerning the computation of $\w^*(\w)$ proved in section 2.

\begin{theorem}\label{sop2}
Let $T = \sk[|\vec{X}{1}{d}|]$ be a power series ring over a field $\sk$. Let $\vec{f}{1}{d}$ be a system of parameters and $R = T/(\vec{f}{1}{d})^n$. Then $\w^*(\w) = (\vec{f}{1}{d})^{n-1}/(\vec{f}{1}{d})^n$, where $\w$ is the canonical module of $R$. 
\end{theorem}

\begin{proof}
Let $T' = \sk[[\vec{f}{1}{d}]]$, $\md = (\vec{f}{1}{d})^nT'$ and $R' \simeq T'/\md^n$. By Theorem \ref{max2}, $\w_{R'}^*(\w_{R'}) = \md^{n-1}/\md^n$. Therefore, since $T$ is free over $T'$, by Corollary \ref{C0}, $\w^*(\w) = \md^{n-1}T/\md^nT = (\vec{f}{1}{d})^{n-1}/(\vec{f}{1}{d})^n$.
\end{proof}

\noindent
{\it Proof of Theorem \ref{sop}.} By Theorem \ref{sop1}, $g(R) \leq \lm(T/(\vec{f}{1}{d}^{n-1}))$. The other inequality follows from Theorem \ref{sop2} which can be seen as follows:

Let $\w$ be the canonical module of $R$. We know that $g(R) \geq \lm(R/\w^*(\w))$ by the fundamental inequalities. This yields $g(R) \geq \lm(T/(\vec{f}{1}{d})^{n-1})$ since $R = T/(\vec{f}{1}{d})^n$ and $\w^*(\w) =  (\vec{f}{1}{d})^{n-1}/(\vec{f}{1}{d})^{n}$. This gives us the equality $g(R) = \lm(T/(\vec{f}{1}{d})^{n-1})$ proving the theorem.\hfill{$\square$}\\

\begin{corollary}
Let $T = \sk[|\vec{X}{1}{d}|]$ be a power series ring over a field $\sk$ of characteristic zero. Let $\vec{f}{1}{d}$ be a system of parameters and $R = T/(\vec{f}{1}{d})^n$. Then $g(R) \leq \lm(R/\soc(R))$.
\end{corollary}

\begin{proof}
We have $\lm(R/\soc(R)) \geq   \lm(T/(\vec{f}{1}{d})^{n-1}) = g(R)$, since $(\vec{f}{1}{d})^n:_T (\vec{X}{1}{d}) \inc (\vec{f}{1}{d})^n :_T (\vec{f}{1}{d}) = (\vec{f}{1}{d})^{n-1}$. 
\end{proof}

\begin{remark}{\rm
Let $T = \sk[[X,Y]]$, $R = T/(X,Y)^n$ and $S = T/(X^n,Y^n)$. Then $S$ is a Gorenstein Artin local ring mapping onto $R$ such that $\lm(S) - \lm(R) = \lm(T/\m_T^{n-1}) = \lm(R/\soc(R))$. This, together with Corollary \ref{C1}, shows that $g(R) = \lm(R/\m^{n-1})$ without any assumptions on the characteristic of $\sk$. Thus, when $d = 2$, using the technique described in Remark \ref{R}, we see that the conclusion of Theorem \ref{sop} is independent of the characteristic of $\sk$.
}\end{remark}

\begin{remark}\label{max}{\rm
By taking $\md$  to be the maximal ideal in Theorem \ref{sop}, we get the following: Let $\sk$ be a field of characteristic zero and $T = \sk[|\vec{X}{1}{d}|]$ be a power series ring over $\sk$. Let  $\m_T = (\vec{X}{1}{d})$ be the maximal ideal of $T$ and $R := T/\m_T^n$. Then $$g(R) = \lm(T/\m_T^{n-1}) = \lm(R/\soc(R)).$$
This also follows immediately from Theorems \ref{max2} and \ref{max1}.
}\end{remark}

\begin{remark}{\rm
If $R = \sk[\vec{X}{1}{d}]/(\vec{X}{1}{d})^n$, where $\sk$ is a field of characteristic zero, it follows from Theorem \ref{max2} and Theorem \ref{max1} that $g(R) = \lm(R/\w^*(\w))$. Thus Question 3.9 in \cite{A} has a positive answer, i.e., in this case, $$\min\{\lm(R/\ma):\ma \simeq \ma^\vee \} = g(R).$$ 
}\end{remark}

\section{\large Applications to Gorenstein Linkage}

\begin{proposition}\label{P2}
Let $(S, \m, \sk)$ be a graded Gorenstein Artin local ring such that $\deg(\soc(S)) = t$. Let $f \in \m$ be a homogeneous element in $S$ of degree $s$ and $\mc = (0 :_S f)$. Then $(\mc :_S \m^n) = \m^{(t+1)-(s+n)} + \mc$.
\end{proposition}

\begin{proof}
Note that $\m^{(t+1)-(s+n)} \cdot \m^n \cdot f \inc \m^{t + 1} = 0$. Hence $\m^{(t+1)-(s+n)} + \mc \inc \mc_n :_S \m^n$. To prove the other inclusion, let $g$ be a homogeneous form of degree less than $(t+1)-(s+n)$. Then $g \cdot f$ is a homogeneous form of degree $t - n$ or less. If $g \cdot f = 0$, then $g \in \mc$. If  $g \cdot f \neq 0$, since $S$ is Gorenstein, there is some element $h \in \m^n$ such that $(g f) \cdot h$ generates \soc(S) and hence is not zero. Thus $g f \m^n \neq 0$ for $g \not\in \m^{(t+1) - (s+n)} + \mc$. Therefore $(\mc :_S \m^n) \inc \m^{(t+1)-(s + n)} + \mc$, proving the proposition.
\end{proof}

\begin{corollary}\label{C2}
Let $\sk$ be a field of characteristic zero and $T = \sk[|\vec{X}{1}{d}|]$ be a power series ring. Let $\m = (\vec{X}{1}{d})$ and $\mc_n = ((\vec{X^n}{1}{d}):_T l^{s})$, where $l = (X_1 + \cdots + X_d)$ and $s \geq (d-2)(n-1) - 1$. Then $(\mc_n :_T \m^n) = \m^{(d-1)(n-1) - s}$.
\end{corollary}

\begin{proof}
By taking $S = T/(\vec{X^n}{1}{d})$, it follows from Proposition \ref{P2} that $(\mc_n :_T \m^n) = \m^{(d-1)(n-1) - s} + \mc_n$. It remains to prove that $\mc_n \inc \m^{(d-1)(n-1) - s}$. 

Let $f$ be a homogeneous element of $T$ such that $f\in \mc$, i.e., $f \cdot l^s \inc (\vec{X^n}{1}{d})$. Hence by Theorem \ref{R3}, $deg(f) \geq \frac{(d(n-1) - s + 1)}{2} \geq (d-1)(n-1) - s$ by the hypothesis on $s$. This shows that $\mc_n \inc m^{(d-1)(n-1) - s}$.
\end{proof}

Let $T = \sk[|\vec{X}{1}{d}|]$ be a power series ring over a field $\sk$. Let $\vec{f}{1}{d}$ be a system of parameters. Let $T' = \sk[|\vec{f}{1}{d}|]$, $\md = (\vec{f}{1}{d})^nT'$ and $\mc_n =(\vec{f^n}{1}{d}):_{T'}l^{s}$, where $l = f_1 + \cdots + f_d$ and $s \geq (d-2)(n-1)-1$. Since, by Corollary \ref{C2}, $(\mc_n :_{T'} \md^n) = \md^{(d-1)(n-1)-s}$ in $T'$, the same holds in $T$ by Remark \ref{R}. Therefore $(\mc_n T :_T \md^n T) = \md^{(d-1)(n-1)-s}T$. Thus we see that 

\begin{proposition}\label{P3}
Let $\sk$ be a field of characteristic zero and $T = \sk[|\vec{X}{1}{d}|]$ be a power series ring. Let $\md = (\vec{f}{1}{d})$, where $\vec{f}{1}{d}$ form a system of parameters. Let $l = f_1 + \cdots  + f_d$ and $s \geq (d-2)(n-1)-1$. Then $\mc_n = ((\vec{f^n}{1}{d}):_T l^s)$ is a Gorenstein ideal such that $(\mc_n :_T \md^n) = \md^{(d-1)(n-1)-s}$.
\end{proposition}

\begin{definition} Let $(T,\m_T,\sk)$ be a regular local ring. An unmixed ideal $\mb \inc T$ is said to be in the {\it {\underline G}orenstein \underline {li}nkage {\underline c}lass of a {\underline c}omplete {\underline i}ntersection (glicci)} if there is a sequence of ideals $\mc_n \inc \mb_n$, $\mb_0 = \mb$, satisfying\\ 1) $T/\mc_n$ is Gorenstein for every $n$\\ 2) $\mb_{n+1} = (\mc_{n} :_T \mb_{n}) $ and \\3) $\mb_n$ is a complete intersection for some $n$.
\end{definition}

We say that $\mb$ is {\it linked} to $\mb_n$ via Gorenstein ideals in $n$ steps.

\begin{remark}\label{R2}\hfill{}\\{\rm
1. Let $\sk$ be a field of characteristic zero and $T = \sk[|\vec{X}{1}{d}|]$ be a power series ring. Let $\md = (\vec{f}{1}{d})$, where $\vec{f}{1}{d}$ form a system of parameters. In \cite{KMMNP}, Kleppe, Migliore, Miro-Roig, Nagel and Peterson show that $\md^n$ can be linked to $\md^{n-1}$ via Gorenstein ideals in 2 steps and hence to $\md$ in $2(n-1)$ steps. But in Proposition \ref{P3}, by taking $s = (d-2)(n-1) $, we see that $\md^n$ can be linked directly via the Gorenstein ideal $(\vec{f^n}{1}{d}):_{T'}l^{(d-2)(n-1)}$ to $\md^{n-1}$, and hence to $\md$, a complete intersection, in $n-1$ steps.

\noindent
2. In a private conversation, Migliore asked if this technique will show that $\md^n$ is self-linked. We see that this can be done by taking $s = (d-2)(n-1) - 1$ in Proposition \ref{P3}. Thus $\md^n$ is linked to itself via the Gorenstein ideal $(\vec{f^n}{1}{d}):_{T'}l^{(d-2)(n-1) - 1}$.
}\end{remark}

\vskip 5 pt
\centerline{\bf A Possible Approach to the Glicci Problem}

\noindent
{\bf The Glicci problem:} Given any homogeneous ideal $\mb \inc T := \sk[\vec{X}{1}{d}]$, such that $R := T/\mb$ is Cohen-Macaulay, is it true that $\mb$ is glicci?
  
A possible approach to the glicci problem is the following: Choose $\mc_n \inc \mb_n$ to be the closest Gorenstein. The question is: Does this ensure that $\mb_n$ is a complete intersection for some $n$? 

\begin{example}{\rm
Let $T = \sk[|\vec{X}{1}{d}|]$, where $\char(\sk) = 0$. Let $\md = (\vec{f}{1}{d})$ be an ideal generated minimally by a system of parameters. We know by Theorems \ref{sop} and \ref{sop1} that the ideal $\mc_n = (\vec{f^n}{1}{d}):_T (f_1 + \cdots + f_d)^{(d-2)(n-1)}$ is a Gorenstein ideal closest to $\md^n$. Now by taking $s = (d-2)(i-1)$ in Proposition \ref{P3}, we see that $\mc_i :_T \md^i = \md^{i-1}$, $2 \leq i \leq n$. Thus $\md^n$ can linked to $\md$ by choosing a closest Gorenstein ideal at each step.
}\end{example}

\section{\large The Codimension Two Case}

We begin this section by recalling the following result of Serre characterizing Gorenstein ideals of codimension two. 

\begin{remark}{\rm
Let $(T,\m_T,\sk)$ be a regular local ring of dimension two. Let $\mc$ be an $\m_T$ primary ideal such that $S = T/\mc$ is a Gorenstein Artin local ring. Then $S$ is a complete intersection ring, i.e., $\mc$ is generated by 2 elements. 
}\end{remark}

\noindent
{\bf Notation:} For the rest of this section, we will use the following notation: Let $(T,\m_T,\sk)$ be a regular local ring of dimension 2, where that $\sk$ is infinite. By $\mu(\-)$, we denote the minimal number of generators of a module and by $e_0(\-)$, we denote the multiplicity of an $\m_T$-primary ideal in $T$. For an ideal $\mb$ in $T$, by $\bar{\mb}$, we denote the integral closure of $\mb$ in $T$.

\begin{remark}{\rm  We state the basic facts needed in this section in this remark. Their proofs can be found in \cite{HS} (Chapter 14).

\noindent
1. Let $\mb$ be an $\m_T$-primary ideal. We define the order of $\mb$ as $\ord(\mb) = \max\{ i : \mb \inc \m_T^i\}$.

Since $\m_T$ is integrally closed, $\ord(\mb) = \ord(\overline \mb)$.

\noindent
2.  Let $\mb$ be an $\m_T$-primary ideal. Since $\sk$ is infinite, a minimal reduction of $\mb$ is generated by 2 elements.  

Further, if $\mc$ is a minimal reduction of $\mb$, the multiplicity of $\mb$, $e_0(\mb) = \lm(T/\mc)$.

\noindent
3. The product of integrally closed $\m_T$-primary ideals is integrally closed. In particular, if $\mb$ is an integrally closed $\m_T$-primary ideal, then so is $\mb^n$ for each $n \geq 2$.  

\noindent
4. For an $\m_T$-primary ideal $\mb$, $\lm((\mb:\m_T)/\mb) = \mu(\mb) -1 \leq \ord(\mb)$.  Further, if $\mb$ is integrally closed, $\mu(\mb) -1 = \ord(\mb)$. 

In particular, this yields $\mu(\mb) \leq \mu(\overline \mb)$.
}\end{remark}

\begin{proposition}\label{P4}
Let $(T,\m_T,\sk)$ be a regular local ring of dimension two and let $\mb$ be an $\m_T$-primary ideal.  The closest (in terms of length) Gorenstein ideals contained in $\mb$ are its minimal reductions.
\end{proposition}

\begin{proof} 
Let $\mc \inc \mb$ be any Gorenstein ideal (and hence a complete intersection by the above remark). It is easy to see that $\lm(T/\mc) \geq \lm(T/(f,g))$, where $(f,g) \inc \mb$ is a minimal reduction of $\mb$. The reason is that \[\begin{array}{rll}\lm(T/\mc) &= e_0(\mc)&\\ &\geq e_0(\mb)&\text{ since }\mc \inc \mb\\& = \lm(T/(f,g)).&\end{array}\]  As a consequence, $$\lm(T/\mc) - \lm(T/\mb) \geq \lm(T/(f,g)) - \lm(T/\mb),$$ $$\text{i.e., }\lm(\mb/\mc) \geq \lm(\mb/(f,g)).$$ Thus the closest Gorenstein ideal contained in $\mb$ is a minimal reduction $(f,g)$.
\end{proof}

\noindent 
We now prove the following theorem which shows that $g(R) \leq \lm(R/\soc(R))$ where $R$ is the Artinian quotient of a 2-dimensional regular local ring.

\begin{theorem}\label{T}
Let $(T,\m_T,\sk)$ be a regular local ring of dimension 2, with infinite residue field $\sk$. Set $R = T/\mb$ where $\mb$ is an $\m_T$-primary ideal. Then $g(R) \leq \lm(R/\soc(R))$, i.e., there is a Gorenstein ring $S$ mapping onto $R$ such that $\lm(S) - \lm(R) \leq \lm(R/\soc(R))$.
\end{theorem}

In order to prove Theorem \ref{T}, we use a couple of formulae for $e_0(\mb)$ and $\lm(R)$ (which can be found, for example, in \cite{JV}). We need the following notation. 

Let $(T,\m)$ and $(T',\n)$ be two-dimensional regular local rings. We say that $T'$ birationally dominates $T$ if $T \inc T'$, $\n \cap T = \m$ and $T$ and $T'$ have the same quotient field. We denote this by $T \leq T'$. Let $[T':T]$ denote the degree of the field extension $T/\m \inc T'/\n$.

Further if $\mb$ is an $\m$-primary ideal in $T$, let $\mb^{T'}$ be the ideal in $T'$ obtained from $\mb$ by factoring $\mb T' = x \mb^{T'}$, where $x$ is the greatest common divisor of the generators of $\mb T'$. 
The following theorem (\cite{JV}, Theorem 3.7) gives a formula for $e_0(\mb)$.

\begin{theorem}\label{multiplicityFormula}
Let $(T,\m_T,\sk)$ be a two-dimensional regular local ring and $\mb$ be an $\m_T$-primary ideal. Then $$e_0(\mb) = \sum_{T \leq T'} [T':T]\ord(\mb^{T'})^2.$$
\end{theorem}

The following formula (\cite{JV}, Theorem 3.10) is attributed to Hoskin and Deligne.

\begin{theorem}[Hoskin-Deligne Formula]\label{HD}
Let $T$, $\mb$ and $R$ be as in Theorem \ref{T}. Further assume that $\mb$ is an integrally closed ideal. Then, $$\lm(R) = \sum_{T \leq T'}\left(\begin{array}{c}\ord(\mb^{T'})+1\\2\end{array}\right)[T':T].$$
\end{theorem}

\begin{corollary}
Let $T$, $\mb$ and $R$ be as in the Hoskin-Deligne formula. Then we have the inequality $$e_0(\mb) + \ord(\mb) \leq 2\lm(R).$$
\end{corollary}

\begin{proof} 
By Theorem \ref{multiplicityFormula}, we have $e_0(\mb) = \sum_{T \leq T'}\ord(\mb^{T'})^2[T':T].$

Using the Hoskin-Deligne formula, we see that $$\lm(R) = \sum_{T \leq T'} \frac{\ord(\mb^{T'})^2 + \ord(\mb^{T'})}{2}[T':T]$$ giving us $$ 2 \lm(R) = e_0(\mb) + \sum_{T \leq T'} \ord(\mb^{T'})[T' : T].$$ Since $T \leq T$ and $\mb^T = \mb$, we get the required inequality. 
\end{proof}

\begin{corollary}
Let $T$, $R$ and $\mb$ be as in Theorem \ref{T}. Then $$e_0(\mb) + \mu(\mb) - 1 \leq 2\lm(T/\mb).$$
\end{corollary}

\noindent
{\bf Proof:} Let $\overline{\mb}$ be the integral closure of $\mb$. By the previous corollary, we have $e_0(\overline{\mb}) + \ord(\overline{\mb}) \leq 2\lm(T/\overline{\mb}).$ Since $\overline{\mb}$ is integrally closed, $\ord(\overline{\mb}) = \mu(\overline{\mb}) - 1$. Thus we get $e_0(\overline{\mb}) + \mu(\overline{\mb}) - 1 \leq 2\lm(T/\overline{\mb}).$ Now $e_0(\mb) = e_0(\overline{\mb})$, $\mu(\mb) \leq \mu(\overline{\mb})$ and $\lm(T/\overline{\mb}) \leq \lm(T/\mb)$, giving the required inequailty.\hfill{$\square$}\\

\noindent
{\bf Proof of Theorem \ref{T}:} For any ideal $\mb$ in $T$, we have $\mu(\mb) - 1 = \lm((\mb:\m)/\m)$. But $(\mb:\m)/\mb \simeq \soc(R)$. Thus by the previous corollary, we have   
$$e_0(\mb) + \lm(\soc(R)) \leq 2 \lm(R).\quad\quad\quad (\sharp\sharp)$$ 

Let $(f,g)$ be a minimal reduction of $\mb$. Then $S := T/(f,g)$ is a complete intersection ring (and hence Gorenstein) mapping onto $R$. Moreover $\lm(S) = e_0(\mb)$. Thus $(\sharp\sharp)$ can be read as $\lm(S) + \lm(\soc(R)) \leq 2 \lm(R).$ Rearranging, we get $\lm(S) - \lm(R) \leq \lm(R) - \lm(\soc(R)).$ This proves that $g(R) \leq \lm(R/\soc(R))$.\hfill{$\square$} \\

\noindent
{\bf Acknowledgement}

\noindent
I would like to thank my advisor Craig Huneke for valuable discussions and encouragement.   

\bibliographystyle{amsplain}

\begin{thebibliography}{100}

\bibitem[1]{A} H. Ananthnarayan, {\it The Gorenstein Colength of an Artinian Local Ring,} Journal of Algebra {\bf 320} (2008), 3438-3446.

\bibitem[2]{HS} C. Huneke, I. Swanson, {\it Integral closure of ideals, rings, and modules,} Cambridge University Press, 2006.

\bibitem[3] {IK} A. Iarrobino, V. Kanev, {\it Power sums, Gorenstein algebras and determinantal loci}, Lecture Notes in Math., 1721, Springer-Verlag, Berlin, 1999.

\bibitem[4] {JV} B. Johnston, J. K. Verma, {\it On the length formula of Hoskin and Deligne and associated graded rings of two-dimensional regular local rings,} Math. Proc. Cambridge Phil. Soc., {\bf 111} (1992), 423 - 432.

\bibitem[5]{KMMNP} J. Kleppe, J. Migliore, R.M. Miro-Roig, U. Nagel, C. Peterson, {\it Gorenstein Liaison, Complete Intersection Liaison Invariants and Unobstructedness,} Mem. Amer. Math. Soc.,  {\bf 154}  (2001),  no. 732, viii+116 pp.

\bibitem[6]{N} D. G. Northcott, {\it Injective Modules and Inverse Polynomials,} J. London Math. Soc. (2), {\bf 8} (1974), 290-296.

\bibitem[7]{RRR} L. Reid, L. Roberts, M. Roitman, {\it On complete intersections and their Hilbert functions,} Canadian Mathematical Bulletin, {\bf 34} (1991), 525-535.

\end{thebibliography}

\end{document}